\documentclass[reqno]{amsart}

\usepackage{amsmath}
\usepackage{amsfonts}
\usepackage{amssymb,enumerate}
\usepackage{amsthm}
\usepackage[all]{xy}
\usepackage{rotating}
\usepackage{hyperref}
\usepackage{color}

\theoremstyle{plain}
\newtheorem{lem}{Lemma}[section]
\newtheorem{cor}[lem]{Corollary}
\newtheorem{prop}[lem]{Proposition}
\newtheorem{thm}[lem]{Theorem}

\newtheorem*{mthm*}{Main Theorem}

\theoremstyle{definition}
\newtheorem{defn}[lem]{Definition}
\newtheorem{ex}[lem]{Example}

\newtheorem{para}[lem]{}

\newtheorem{convention}[lem]{Convention}

\newtheorem*{convention*}{Convention}

\newcommand{\id}{\operatorname{id}}

\newcommand{\Ext}{\operatorname{Ext}}	
\newcommand{\Tor}{\operatorname{Tor}}

\newcommand{\HH}{\operatorname{H}}

\newcommand{\End}{\operatorname{End}}

\newcommand{\Ker}{\operatorname{Ker}}

\newcommand{\bbz}{\mathbb{Z}}

\newcommand{\xra}{\xrightarrow}
\newcommand{\xla}{\xleftarrow}

\renewcommand{\geq}{\geqslant}
\renewcommand{\leq}{\leqslant}

\def\Tor{\operatorname{Tor}}
\def\Ext{\operatorname{Ext}}

\def\m{\mathfrak{m}}
\def\fm{\mathfrak{m}}

\def\spec{\operatorname{Spec}}

\def\End{\mathrm{End}}

\numberwithin{equation}{lem}

\begin{document}

\bibliographystyle{amsplain}

\title[Semi-fiber products of algebras and lifting of complexes]{Semi-fiber products of algebras\\ and lifting of complexes}

\author{Saeed Nasseh}
\address{Department of Mathematical Sciences\\
Georgia Southern University\\
Statesboro, GA 30460, U.S.A.}
\email{snasseh@georgiasouthern.edu}

\author{Maiko Ono}
\address{Department of Mathematics, Okayama University, 3-1-1 Tsushima-naka, Kita-ku, Okayama 700-8530, Japan}
\email{onomaiko.math@okayama-u.ac.jp}

\author{Yuji Yoshino}
\address{Graduate School of Environmental, Life, Natural Science and Technology, Okayama University, Okayama 700-8530, Japan}
\email{yoshino@math.okayama-u.ac.jp}

\thanks{M. Ono was partly supported by the Wesco Scientific Promotion Foundation and Y. Yoshino was supported by JSPS Kakenhi Grant 19K03448 and 24K0669.}


\keywords{Fiber product, lifting, Poincar\'{e} series, rational point, retraction, sections, semi-fiber product, tensor algebra, trivial extension}
\subjclass[2020]{13B35, 13D02, 13E05}

\begin{abstract}
Let $k$ be a field. In this paper, we define the notion of semi-fiber products of commutative $k$-algebras and show that the class of such rings contains several classes of commutative rings, including that of the fiber products of local $k$-algebras over their common residue field $k$. For a noetherian local $k$-algebra $R$ and an ideal $I$ of $R$, under certain conditions, we characterize the liftability of $k$ along the natural surjection $R\twoheadrightarrow R/I$ in terms of retractions, sections, and the existence of semi-fiber product decompositions of $R$.  
\end{abstract}

\maketitle


\section{Introduction}\label{sec20231210a}

\begin{convention}\label{para20250907a}
Throughout the paper, all rings are commutative with unity and $k$ is a field. By the notation $(R,\fm_R)$ we mean a (not necessarily local or noetherian) $k$-algebra $R$ with a specified maximal ideal $\m_R$ for which the composition map $k \hookrightarrow  R \twoheadrightarrow  R/\m_R$ is an isomorphism, i.e., $\fm_R$ is a $k$-rational point of $\spec(R)$. In this situation, we simply refer to $(R,\fm_R)$ as a ``$k$-algebra''. When we say that a $k$-algebra $(R,\fm_R)$ is complete, we mean it is complete in the $\fm_R$-adic topology; in this case, note that $(R,\frak m_R)$ is a complete local ring. Also, for a $k$-algebra homomorphism $f\colon (R,\fm_R)\to (S,\fm_S)$ we always assume $f(\m_R) \subseteq \m_S$. When there is no fear of confusion, we may simply write $f\colon R\to S$ is a $k$-algebra homomorphism without indicating the $k$-rational points $\fm_R$ and $\fm_S$.
\end{convention}

Typical examples of $k$-algebras of particular interest include complete local rings with coefficient field $k$ and affine rings over the field $k$ with specified $k$-rational points. In this paper, we combine three topics: (1) semi-fiber products of $k$-algebras; (2) lifting theory of modules and complexes; and (3) the existence of sections and retractions for certain $k$-algebra homomorphisms, as we describe subsequently.\vspace{2mm}

In Section~\ref{section20240729a}, we define the notion of semi-fiber products of $k$-algebras, which are commutative rings with unity. We show that the class of such rings encompasses fiber product rings of local $k$-algebras over their common residue field $k$, trivial extensions of $k$-algebras by their modules, tensor algebras, and complete tensor algebras. This section also lists some basic properties of semi-fiber products and provides a characterization of such rings in Propositions~\ref{prop20250924a} and~\ref{pro20251105a}.\vspace{2mm}

Section~\ref{sec20251109a} of this paper is devoted to the lifting theory of modules and complexes; a topic that has been studied intensively by Auslander, Ding, Solberg~\cite{ADS} and Yoshino~\cite{yoshino} and is one of the most fundamental and applied concepts in commutative algebra. This theory has been recently generalized to the differential graded (DG) homological algebra setting for the purpose of studying the deformations of finitely generated modules over local rings; see~\cite{NOY-jadid, NOY3, NOY2, NOY1, NOY, nasseh:lql, nassehyoshino, OY}.\vspace{2mm}

In this paper, for a local $k$-algebra $R$ and an ideal $I$ of $R$, we are particularly interested in the liftability of the $R/I$-module $k$ along the natural $k$-algebra surjection $R\twoheadrightarrow R/I$. This is the point where the notions of semi-fiber products, lifting theory, and sections and retractions of certain $k$-algebra homomorphisms merge. Our main result in this paper is the following for which the proof is given in Section~\ref{sec20251109b}.

\begin{mthm*}\label{Main Theorem}
Assume that $(R, \m_R)$ is a noetherian local $k$-algebra and $(T,\frak m_T)$ is a complete $k$-algebra. Let $\varphi\colon (T, \m_T)\to (R, \m_R)$ be a flat $k$-algebra homomorphism, and let $\overline{R}= R / \m_TR$ be the fiber ring of $\varphi$. Then, the following statements are equivalent:
\begin{enumerate}[\rm(i)]
\item 
The $\overline{R}$-module $k$ is liftable to $R$;
\item
$\varphi$ has a retraction; 
\item
There is an ideal $\frak a$ of $R$ such that, setting $\fm_S=\frak a$ and $S=k\oplus \fm_S$, the $k$-algebra $R$ admits a semi-fiber product decomposition $R\cong T \ltimes _k S$.
\end{enumerate}
\end{mthm*}

As we explicitly state in Corollaries~\ref{cor20251018a} and~\ref{cor20250928a}, examples that fall into the situation of Main Theorem include the cases where:
\begin{enumerate}[\rm(a)]
\item
$T=k[\![\underline{x}]\!]$ with $\fm_T=(\underline{x})$, where $\underline{x}$ is an $R$-regular sequence and $R$ is complete;
\item
$T=k[x]$ with $\fm_T=(x)$, where $x\in \fm_R$ satisfies the condition $(0:x)=x^nR\neq (0)$ for a positive integer $n$.
\end{enumerate}\vspace{2mm}

Finally, it is worth mentioning that a variation of Main Theorem is proved in Section~\ref{sec20251109a}; see Corollary~\ref{cor20251031a}. This result is particularly interesting because it includes a necessary and sufficient condition for the liftability of $k$ in terms of sections, while our Main Theorem involves such a condition in terms of retractions.

\section{Semi-fiber products of algebras}\label{section20240729a}

In this section we define the notion of semi-fiber products of algebras and discuss some of their properties. We show that the class of semi-fiber products includes several classes of commutative rings. Furthermore, a characterization of semi-fiber products that might be of independent interest is given in Proposition~\ref{pro20251105a}.  

\begin{para}\label{disc20250911a}
For a $k$-algebra $(R, \m _R)$ we have the equality $R = k  \oplus  \m_R$ of $k$-vector spaces. 
Therefore, every element of $R$ is uniquely written as a sum of elements of $k$ and $\m_R$, that is, for every $r\in R$ there exist $\ell\in k$ and $x\in \fm_R$ such that
\begin{equation}\label{eq20250909b}
r=\ell\oplus x.
\end{equation}
Here, for clarity, we use the notation $\oplus$ for the addition of elements in a direct sum to distinguish it from the usual addition in a ring. For instance, note that for another element $r'=\ell'\oplus x'$ of $R$ with $\ell'\in k$ and $x'\in \fm_R$, the multiplication $rr'$ in $R$ translates to $rr'=\ell\ell'\oplus (\ell x'+\ell' x+xx')$, where $+$ is the usual addition in $R$.
\end{para}

In the next discussion, we lay out the foundation for the definition of semi-fiber products. Throughout the paper, we will automatically employ the assumptions and notation of this discussion whenever we talk about semi-fiber products.

\begin{para}\label{assumption}
Let $(R, \m_R)$ and $(S, \m_S)$ be $k$-algebras and assume that there is an $R$-action on $\m_S$, denoted by $\ast$, under which  $\m_S$ becomes an $(R, S)$-bimodule. In this circumstances, since $S$ is a $k$-algebra, for all elements $r=\ell\oplus x$ of $R$ as in~\eqref{eq20250909b}, $y \in \m_S$, and $s \in S$ we have the equalities
\begin{gather*}
\ell\ast y=\ell y\\
r\ast y=\ell y+x\ast y\\
r {\ast} (ys) = (r {\ast} y)s.
\end{gather*}
Equivalently, there exists a $k$-linear ring homomorphism $R \to \End _S (\m_S)$.
\end{para}

\begin{defn}\label{defn20251109a}
Consider the setting of~\ref{assumption}. We define the \emph{semi-fiber product} of $R$ by $S$, denoted $R \ltimes _k S$, as follows: 
\begin{enumerate}[\rm(a)]
\item
As a $k$-vector space, we set $R \ltimes _k S = k \oplus \m_R \oplus \m_S$. Note that, by~\ref{disc20250911a}, we have
\begin{align}
R \ltimes _k S&=R\oplus \fm_S\label{eq20250913a}\\
&=\fm_R\oplus S\notag
\end{align}
as $k$-vector spaces.
\item
For all elements $r,r'\in R$ and $y, y' \in \m_S$, considering the equality~\eqref{eq20250913a}, we can define a multiplication in $R \ltimes _k S$ by the formula
\begin{equation}\label{eq20250909a}
(r\oplus y) \cdot (r'\oplus y') =rr'\oplus \left(r {\ast} y' + r' {\ast}y + yy'\right).
\end{equation}
\end{enumerate}
Setting $\m _{R \ltimes _k S}= \m_R \oplus \m_S$, one can check that $R \ltimes _k S$  is a commutative ring with unity and hence, $(R \ltimes _k S, \m _{R \ltimes _k S})$ is a $k$-algebra. 
It follows from part (a) that both  $(R, \m_R)$ and $(S, \m_S)$ are naturally $k$-subalgebras of $(R \ltimes _k S, \m _{R \ltimes _k S})$.
\end{defn}

\begin{defn}
We say that a $k$-algebra $(A,\fm_A)$ \emph{admits a semi-fiber product decomposition} if there exist $k$-subalgebras $(R, \m_R)$ and $(S, \m_S)$ of $A$ such that an isomorphism $A\cong R \ltimes _k S$ of $k$-algebras holds.
\end{defn}

A universal property of the semi-fiber products can be described as follows.

\begin{para}
Let $(R, \m_R)$ and $(S, \m_S)$ be as in~\ref{assumption}. For a $k$-algebra $(T, \m_T)$, assume that there are a $k$-algebra homomorphism  $f\colon R \to T$ (hence, $\m_T$ is an $R$-module via $f$) and an $R$-module homomorphism $g\colon \m_S \to \m_T$  that preserves multiplication (that is, 
$g(yy') = g(y)g(y')$ for all elements $y, y' \in \m_S$). 
Then, there is a unique $k$-algebra homomorphism $\varphi\colon R \ltimes _k S \to T$ satisfying the equalities
$$
\varphi |_R = f\qquad \text{and}\qquad \varphi |_{\m_S} =g.
$$
In fact, for each element $r\in R$ and $y\in \fm_S$, the map $\varphi$ is defined by $\varphi\left(r\oplus y\right)=f(r)+g(y)$, which is indeed a $k$-algebra homomorphism using the equality~\eqref{eq20250909a}.
%
\end{para}

The class of semi-fiber products of algebras includes several classes of commutative rings that are listed in the following examples. 

\begin{ex}\label{ex20251013a}
Let $(R,\frak m_R, k)$ and $(S,\frak m_S,k)$ be commutative noetherian local rings with a common residue field $k$. Then, the \emph{fiber product} of $R$ and $S$ is defined to be
$$
R\times_kS=\{(r,s)\in R\times S\mid \pi_R(r)=\pi_S(s)\}
$$
where $R\xra{\pi_R} k\xla{\pi_S}S$ are the natural surjections. Note that $R\times_k S$ is a commutative noetherian local ring with maximal ideal $\fm_{R\times_k S}=\fm_R\oplus \fm_S$ and residue field $k$. Note also that the class of such fiber product rings coincides with the class of commutative noetherian local rings with decomposable maximal ideal; we refer the reader to \cite{dress, freitas, Geller, kostrikin, moore, NOY-jadidtar, nasseh:vetfp, NST, nasseh:ahplrdmi, nasseh:lrqdmi, ogoma:edc} for the properties and applications of fiber product rings.

Assume further that $R$ and $S$ are $k$-algebras. Defining the $R$-action $\ast$ on $\fm_S$ by the equality $\m_R {\ast} \m_S =0$, we see that $R \times _k S\cong R \ltimes _k S$. More precisely, under this isomorphism, every element $\ell\oplus x\oplus y$ in $R \ltimes _k S$ with $\ell\in k$, $x\in \fm_R$, and $y\in \fm_S$ corresponds to the element $(\ell\oplus x,\ell\oplus y)$ in $R\times_k S$. Hence, every fiber product of $k$-algebras of the above form is a semi-fiber product, as the names suggest.
\end{ex}

\begin{ex}\label{ex20251107a}
The \emph{trivial extension} (also known as \emph{Nagata's idealization}) of a ring $R$ by an $R$-module $M$ is defined to be the $R$-module $R\oplus M$ that is equipped with a ring structure given by the multiplication $(r\oplus m)\cdot (r'\oplus m')=rr'\oplus (rm'+r'm)$ for all $r,r'\in R$ and $m,m'\in M$. Note that $M$ is an ideal of $R\ltimes M$ with $M^2=0$.

For a $k$-algebra $(R, \m_R)$ and an $R$-module $M$, setting $S = k \ltimes M$, where $M$ is considered as a $k$-vector space (hence, $(S, \m_S)$ is a local $k$-algebra with the maximal ideal $\m_S = M$), we have a natural isomorphism $R \ltimes M\cong R \ltimes _k S$ of $k$-algebras. Hence, every trivial extension of a $k$-algebra by a module is a semi-fiber product.
\end{ex}

\begin{ex}\label{ex20250923a}
Let $(R, \m_R)$ and $(T, \m_T)$ be $k$-algebras. Note that the tensor product $(R \otimes _k T, \m_R \otimes _k T + R \otimes _k \m_T)$ is a  $k$-algebra, which we refer to as a tensor algebra. 
Now, setting $S= k\oplus \m_S$, where $\m_S= R \otimes _k \m_T$ (hence, $\m_S$ is an $R$-module), we have that $(S, \m_S)$ is a $k$-algebra and there is an isomorphism $R  \otimes _k T\cong R \ltimes _k S$ of $k$-algebras. Therefore, every tensor algebra is a semi-fiber product. Note that $(S, \m_S)$ is not necessarily noetherian; see Example~\ref{ex20250924a}.
\end{ex}

\begin{ex}\label{ex20251104a}
Let $(R, \m_R)$ and $(T, \m_T)$ be complete local $k$-algebras. Similar to Example~\ref{ex20250923a}, every complete tensor product $k$-algebra $(R \widehat{\otimes} _k T, \m_R \widehat{\otimes} _k T + R \widehat{\otimes} _k \m_T)$ is a semi-fiber product. More precisely, in this case we have $R  \widehat{\otimes} _k T\cong R \ltimes _k S$, where $S= k \oplus \m_S$ and $\m_S= R \widehat{\otimes} _k \m_T$.
\end{ex}

In Definition~\ref{defn20251109a}, if the $R$-action $\ast$ on $\fm_S$ comes through a homomorphism of commutative noetherian local $k$-algebras, then the notion of semi-fiber product degenerates to that of the ordinary fiber product introduced in Example~\ref{ex20251013a}. We record this in the following result.

\begin{prop}\label{induced by f}
Let $(R, \m_R)$ and $(S, \m_S)$ be commutative noetherian local $k$-algebras. If the $R$-module structure on $\m_S$ is induced through a $k$-algebra homomorphism $f\colon R \to S$, then the map $\psi\colon R \ltimes _k S \to R \times _k S$ defined by the formula 
$$
\psi(\ell \oplus x \oplus y)=(\ell \oplus x,\ell\oplus (y + f(x)))
$$
with $\ell \in k$, $x \in \m_R$, and $y \in \m_S$, is an isomorphism of $k$-algebras.
\end{prop}

\begin{proof}
The fact that $\psi$ is a $k$-algebra homomorphism can be checked easily using the equality~\eqref{eq20250909a}. Also, for all $(\ell\oplus x,\ell\oplus y)$ in $R\times_kS$ note that
$$
\psi(\ell \oplus x \oplus (y-f(x)))=(\ell \oplus x,\ell\oplus y)
$$
and therefore, $\psi$ is surjective. The fact that $\psi$ is injective is obvious. 
\end{proof}

The next result records some of the basic facts about semi-fiber products. The reader might be curious to know what other ring theoretic, homological, or geometric properties semi-fiber products have, i.e., blanket rules that cover all of the Examples~\ref{ex20251013a} - \ref{ex20251104a} simultaneously. At this time, that topic is not our main focus and of interest in this paper.  

\begin{prop}\label{prop20250924a}
Let $(R, \m_R)$ and $(S, \m_S)$ be as in~\ref{assumption}. Then, the following statements hold:

\begin{enumerate}[\rm(a)]
\item\label{item20250923a}
If $R$ and $S$ are finitely generated $k$-algebras, then so is $R \ltimes_k S$. 
\item\label{item20250923b}
If $R$ and $S$ are localizations of finitely generated $k$-algebras, then so is $R \ltimes_k S$.
\item\label{item20250923c}
Assume that $(R,\fm_R)$ and $(S,\fm_S)$ are complete local $k$-algebras and that the $\m_R$-adic topology on $\m_S$ is separated, that is, $\bigcap _{n=1}^\infty (\m_R^n\ast \m_S) =(0)$. Then, $R \ltimes_k S$ is a complete local $k$-algebra.
\item\label{item20250923d}
If $R$ is noetherian and $\m_S$ is finitely generated as an $R$-module, then $R \ltimes_k S$ is also noetherian. 
\item\label{item20250923e}
If $R \ltimes_k S$ is noetherian, then so is $R$.
\end{enumerate}
\end{prop}

We prove only part~\eqref{item20250923c} as the proofs of the other statements are straightforward.

\begin{proof}
To prove part~\eqref{item20250923c}, let  $A = R \ltimes_k S$ and  $\m_A = \m_R \oplus \m_S$. 
Then, we have the equality $\m_A ^2 = \m_R^2 \oplus (\m_R\ast \m_S + \m_S^2)$. More generally, it follows by induction that 
$$
\m_A^n = \m_R^n \oplus \left(  \m_R ^{n-1} \ast \m_S +\m_R ^{n-2} \ast \m_S^2 + \cdots + \m_R\ast \m_S^{n-1} + \m_S^n\right)   
$$
for all integers $n \geq 1$. Since by our assumption the topology on $\m_S$ defined by the ideal sequence $\{\m_R ^{n} \ast \m_S\mid  n \geq 1 \}$ is separated, 
it follows from~\cite[Theorem (30.1)]{Nagata} that for any positive integer $n$ there exists a positive integer $m(n)$ such that the containment $\m_R^{m(n)} \ast \m_S \subseteq \m_S ^n$ holds.
Therefore, for all positive integers $n$ we have 
\begin{equation}\label{eq20250923d}
\m_A ^{m(n)+n} \subseteq \m_R^n \oplus \m_S ^n . 
\end{equation}
If $\mathcal{S}=\{a_n \oplus b_n  \in \m_R \oplus \m_S\mid n \geq 1\}$ is a Cauchy sequence in the $\m_A$-adic topology, then by~\eqref{eq20250923d} the sequences $\{a_n\mid n\geq 1\}$ and $\{b_n\mid n\geq 1\}$ are also Cauchy sequences in $R$ and $S$ in the $\fm_R$-adic and $\fm_S$-adic topologies, respectively. Thus, $\mathcal{S}$ converges by our assumption, that is, $R \ltimes_k S$ is a complete local $k$-algebra.
\end{proof}

Our next example shows that in part~\eqref{item20250923e} of Proposition~\ref{prop20250924a}, the assumption that $R \ltimes_k S$ is noetherian does not necessarily imply that $S$ is noetherian.

\begin{ex}\label{ex20250924a}
Let $A=k[x,y]$, $R=k[x]$, and $S=k[x^ny\mid 0\leq n\in \mathbb{Z}]$. Then, the $k$-algebra $A$ admits a semi-fiber product decomposition
$A \cong R\ltimes_k S$.
\end{ex}

We conclude this section with the following result that provides a simple characterization of semi-fiber products which might be of independent interest.

\begin{prop}\label{pro20251105a}
The following conditions are equivalent for a $k$-algebra  $(A, \m_A)$:
\begin{enumerate}[\rm(i)]
\item
$A$ admits a semi-fiber product decomposition $A \cong R \ltimes _k S$ into $k$-subalgebras $(R, \fm_R)$ and $(S, \fm_S)$;
\item 
There exists a direct sum decomposition  $\m_A=\frak u \oplus \frak I$ of $k$-vector spaces, where $\frak u$ is a multiplicatively closed subset of $A$ and $\frak I$ is an ideal of $A$.
\end{enumerate}
\end{prop}

\begin{proof}
The implication (i)$\implies$(ii) follows readily by setting  $\frak u = \m_R$ and $\frak I=\m_S$.

(ii)$\implies$(i): Let $R = k\oplus \frak u$ and $S = k \oplus \frak I$. Setting $\fm_R=\frak u$ and $\fm_S=\frak I$, note that $(R, \fm_R)$ and $(S, \fm_S)$ are $k$-subalgebras of $(A, \fm_A)$. Defining the $R$-action on $\fm_S$ by the multiplication of elements of $R$ and $\fm_S$ in $A$, we see that $A \cong R \ltimes _k S$.
\end{proof}

\section{Lifting of complexes}\label{sec20251109a}

In this section, we introduce the notion of liftability and provide some examples. Using the Poincar\'e series as a tool, we prove Theorem~\ref{thm20251102b} which enables us to characterize the liftability of $k$ with respect to the notion of sections; see Corollary~\ref{cor20251031a}. We start with the following convention to use throughout this section.

\begin{convention}\label{conv20251106a}
In this section, $(R,\fm_R)$ is a $k$-algebra and for an ideal $I\subseteq \m_R$ of $R$, we denote by $(\overline{R}, \m_{\overline{R}})$ the residue $k$-algebra $R/I$ with $\m_{\overline{R}} = \m_R/I$. Furthermore, we let $\pi$ denote the natural surjection $R\twoheadrightarrow \overline{R}$.
\end{convention}

Next, we remind the reader of the classical definition of lifting.

\begin{defn}\label{para20240616a}
We say that an $\overline{R}$-free complex (i.e., a complex of free $\overline{R}$-modules) $F$ is \emph{liftable} to $R$ (along $\pi$) if there exists an $R$-free complex $L$ such that
$L \otimes _R \overline{R} \cong  F$ as $\overline{R}$-complexes. In this situation, we say that $L$ is a \emph{lifting} of $F$ to $R$.
We call an $\overline{R}$-module $M$
\emph{liftable} to $R$ if there is an $\overline{R}$-free resolution of $M$ that is liftable to $R$.
\end{defn}

\begin{para}\label{para20251104a}
In Definition~\ref{para20240616a}, if $(R,\fm_R)$ is a noetherian complete local ring and $I$ is generated by an $R$-regular sequence, then a bounded below $\overline{R}$-complex $F$ consisting of finitely generated free $\overline{R}$-modules is liftable to $R$ along $\pi$ provided that $\Ext^2_{\overline{R}}(F,F)=0$; see Yoshino~\cite[Lemma (3.2)]{yoshino}. This statement was generalized by Nasseh and Sather-Wagstaff~\cite[Main Theorem]{nasseh:lql} to the case where $I$ is not necessarily generated by an $R$-regular sequence. Such a generalization was possible by replacing $\overline{R}$ by the Koszul complex $K^R(I)$ and assuming that $F$ is a homologically bounded below and homologically degreewise finite semifree DG $K^R(I)$-module.
\end{para}


As we mentioned in the introduction, in this paper we are interested in the liftability of the $\overline{R}$-module $k$ to $R$. Thus, at this point, we proceed by providing some examples in this direction.
In the following examples, $\overline{x}$ and $\overline{y}$ represent the residues of the variables $x$ and $y$ in $\overline{R}$. Also, the symbol $\simeq$ is used for quasi-isomorphisms. Our first example considers the lifting property of the $\overline{R}$-module $k$ along $\pi$, where $R$ is a semi-fiber product which is not a fiber product; see Example~\ref{ex20250924a}.

\begin{ex}
Let $R = k[x, y]$ and $\overline{R} = R/xR\cong k[\overline{y}]$. Then, the $\overline{R}$-free resolution
$0\to \overline{R} \xra{\overline{y}} \overline{R} \to 0$ of the residue field $k$ is liftable to $R$ with the unique (up to quasi-isomorphism) lifting $L: 0 \to R \xra{y} R \to 0$.
Note that $L\simeq R/yR$.
\end{ex}

The next example considers the lifting property of the $\overline{R}$-module $k$ along $\pi$, where $R$ is a fiber product ring.

\begin{ex}\label{ex20251102a}
Let  $R = k[x, y]/(xy) \cong k[x] \times _k k[y]$ and $\overline{R} = R/(x-y) \cong k[\overline{x}]/(\overline{x}^2)$. 
Then, the $\overline{R}$-free resolution
$F: \cdots \to \overline{R} \xra{\overline{x}} \overline{R} \xra{\overline{x}} \overline{R} \xra{\overline{x}} \overline{R} \to 0$
of the residue field $k$ is liftable to $R$ with the liftings
\begin{gather*}
\left(\cdots \to R \xra{x} R \xra{y} R \xra{x} R \to 0\right)\simeq R/xR\\
\left(\cdots \to R \xra{y} R \xra{x} R \xra{y} R \to 0\right)\simeq R/yR.
\end{gather*}
Note that these are the only liftings of $F$ to $R$ up to quasi-isomorphism.
\end{ex}

The next example exhibits some cases for which $k$ is not liftable along $\pi$.

\begin{ex}
Consider the following cases:
\begin{enumerate}[\rm(a)]
\item
$R=k[x, y] /(x^2-y^3)$ and $\overline{R} = R/(y) \cong k[\overline{x}]/(\overline{x}^2)$; or
\item
$R=k[x]$ and $\overline{R} = R/(x^2)\cong k[\overline{x}]/(\overline{x}^2)$.
\end{enumerate} 
In both of these cases, the $\overline{R}$-free resolution
$\cdots \to \overline{R} \xra{\overline{x}} \overline{R} \xra{\overline{x}} \overline{R} \xra{\overline{x}} \overline{R} \to 0$
of the residue field $k$ is  not  liftable to $R$.
\end{ex}

Here is our main result in this section which provides a necessary condition for the liftability of $k$ along $\pi$.

\begin{thm}\label{thm20251102b}
If $(R, \m_R)$ is a noetherian local $k$-algebra and the $\overline{R}$-module $k$ is liftable to $R$, then $I$ is generated by a part of a minimal generating set of $\m_R$.
\end{thm}

The proof of this theorem is given after some preparation.

\begin{para}\label{lemma1}
Let $M$ be a finitely generated $\overline{R}$-module that is liftable to $R$ with a lifting $R$-complex $L$, and let $F$ be an $R$-free resolution of $\overline{R}$. 
Note that $L \otimes _R F$ is a complex of free $R$-modules. Also, as both $R$- and $\overline{R}$-complexes we have
$M\simeq L \otimes _R \overline{R}$. Hence, $M\simeq L \otimes _R F$ as $R$-complexes, and therefore, $L \otimes _R F$ is an $R$-free resolution of $M$.
\end{para}

\begin{defn}
Assume that $(R,\fm_R,k)$ is a noetherian local ring. Recall that the \emph{Poincar\'e series} of a homologically bounded below and homologically degreewise finite $R$-complex $N$ is defined to be $$P^R_N(t)=\sum_{i\in\bbz}\dim_k\left(\Tor^R_i(k,N)\right)t^i\in \mathbb{Z}[\![t]\!].$$
\end{defn}

As an consequence of our discussion in~\ref{lemma1} we get the following result.

\begin{lem}\label{lemma2}
If $(R, \m_R)$ is a noetherian local $k$-algebra, then for a finitely generated $\overline{R}$-module $M$ that is liftable to $R$ along $\pi$ we have the equality
\begin{equation}\label{eq20250926a}
P^R_M (t) = P ^{\overline{R}}_M(t) P^R_{\overline{R}}(t). 
\end{equation}
\end{lem}

\begin{proof}
Consider the setting of~\ref{lemma1}. Since $M\simeq L\otimes_R\overline{R}$ as $\overline{R}$-complexes, it follows from~\cite[Lemma (1.5.3)]{AF} that
\begin{equation}\label{eq20251106a}
P^{\overline{R}}_M (t) = P ^{R}_L(t) P^{\overline{R}}_{\overline{R}}(t)=P ^{R}_L(t). 
\end{equation}
Another application of~\cite[Lemma (1.5.3)]{AF} to the quasi-isomorphism $M\simeq L\otimes_RF$ of $R$-complexes and using the fact that $F$ is an $R$-free resolution of $\overline{R}$ implies that
\begin{equation}\label{eq20251106b}
P^{R}_M (t) = P ^{R}_L(t) P^{R}_{F}(t)=P ^{R}_L(t) P^{R}_{\overline{R}}(t). 
\end{equation}
The equality~\eqref{eq20250926a} now follows by combining~\eqref{eq20251106a} and~\eqref{eq20251106b}.
\end{proof}


\noindent{\emph{Proof of Theorem~\ref{thm20251102b}.}}
Looking at the coefficients of $t$ on both sides of the equality~\eqref{eq20250926a} from Lemma~\ref{lemma2} we obtain the equality
$$
\nu_R(\fm_R)=\nu_{\overline{R}}(\m_{\overline{R}})+\nu_R(I)
$$
where $\nu$ denotes the minimal number of generators. This equality means that if $I$ is minimally generated by $x_1,\ldots,x_t$, then $\overline{x}_1,\ldots,\overline{x}_t\in \m_R/\m_R^2$ are linearly independent. Thus, $I$ is generated by a part of a minimal generating set of $\m_R$.
\qed\vspace{2mm}

Our next result is a consequence of Theorem~\ref{thm20251102b} which introduces some equivalent conditions for the liftability of $k$ along $\pi$. This result is particularly interesting because it includes a necessary and sufficient condition for the liftability of $k$ in terms of sections, while our Main Theorem involves such a condition in terms of retractions. We first remind the reader of the definitions of section and retraction. 

\begin{defn}
Let $f\colon R\to S$ and $g\colon S\to R$ be $k$-algebra homomorphisms such that $f$ is injective, $g$ is surjective, and $g\circ f=\id_R$. In this situation, $f$ is called a \emph{section} of $g$ and $g$ is called a \emph{retraction} of $f$.
\end{defn}


\begin{cor}\label{cor20251031a}
Assume that $(R, \m_R)$ is a noetherian complete local $k$-algebra and that $I$ satisfies the equality $\m_R I =(0)$. Then, the following are equivalent: 
\begin{enumerate}[\rm(i)]
\item
The $\overline{R}$-module $k$ is liftable to $R$; 
\item
Every finitely generated $\overline{R}$-module is liftable to $R$;
\item
The natural surjection $\pi\colon R \twoheadrightarrow \overline{R}$ has a section;
\item
$R$ admits a semi-fiber product decomposition $R \cong \overline{R}\ltimes I \cong \overline{R} \ltimes_k (k\ltimes I)$. 
\end{enumerate}
\end{cor}

Before proving this result, we note the following fact.

\begin{para}\label{disc20251102s}
Let $(R,\fm_R)$ be a noetherian complete local $k$-algebra and $I$ be an ideal of $R$ such that $\fm_RI=(0)$. It is straightforward to check that if $I$ is generated by a part of a minimal generating set of $\fm_R$, then $R\cong (R/I)\ltimes I $ as $k$-algebras. In fact, in this setting, we may assume that $R=k[\![\underline{x},\underline{y}]\!]/J$, where $\underline{x} \cup \underline{y}$ is a finite set of variables, $J$ is an ideal of $k[\![\underline{x},\underline{y}]\!]$ with $J\subseteq (\underline{x},\underline{y})^2$, and $I$ is the ideal of $R$ generated by $\underline{x}$ modulo $J$.
By our assumption, $J$ contains the ideals $(\underline{x})^2$ and $(\underline{x})(\underline{y})$ and one can take a sequence $\underline{f}$ of elements in $k[\![\underline{y}]\!]$ such that 
$J=((\underline{x})^2,(\underline{x})(\underline{y}), \underline{f}).$
Thus, $R/I\cong k[\![\underline{y}]\!]/(\underline{f})$ and 
therefore, $(R/I) \ltimes I \cong R$.
\end{para}

\noindent \emph{Proof of Corollary~\ref{cor20251031a}.}
The implications (iv)$\implies$(iii)$\implies$(ii)$\implies$(i) are trivial.

(i)$\implies$(iv): If the $\overline{R}$-module $k$ is liftable to $R$, then it follows from Theorem~\ref{thm20251102b} that $I$ is generated by a part of a minimal generating set of $\m_R$. Since $\m_R I =(0)$, as we discussed in~\ref{disc20251102s}, we get the isomorphism $R\cong \overline{R}\ltimes I$. Moreover, it follows from Example~\ref{ex20251107a} that $R \cong \overline{R} \ltimes_k (k\ltimes I)$. \qed\vspace{2mm}

We conclude this section with the following remark.

\begin{para}\label{disc20251102r}
Let $(R,\fm_R)$ be a (complete) local $k$-algebra. Assuming that the $\overline{R}$-module $k$ is liftable to $R$, the homomorphism $\pi$ does not necessarily have a section. For example, the $k$-algebra homomorphism $R = k[x, y]/(xy) \twoheadrightarrow \overline{R} = R/(x-y)$ from Example~\ref{ex20251102a} does not admit a section. Thus, the additional assumption that $\fm_RI=(0)$ is necessary in Corollary~\ref{cor20251031a}. In light of this corollary, a general question to ask is: what additional conditions are necessary for $\pi$ to have a section?
\end{para}

\section{Proof of Main Theorem}\label{sec20251109b}

In this section, we prove our Main Theorem and discuss some of its consequences after making the necessary preparation. We start by providing another characterization of semi-fiber products in terms of the retractions of $k$-algebra homomorphisms. This characterization is used in the proof of Main Theorem.

\begin{thm}\label{retraction theorem}
For $k$-algebras $(R, \m_R)\subseteq (A, \m_A)$ the following hold: 
\begin{enumerate}[\rm(a)]
\item 
The inclusion $R \subseteq A$ has a retraction if and only if there is a $k$-subalgebra $(S, \m_S)$ of $A$ such that $A$ admits a semi-fiber product decomposition       
$A\cong R \ltimes _k S$. 
\item
Assume that $R \subseteq A$ has a retraction. Then, $S \subseteq A$ obtained in part (a) also has a retraction if and only if $A \cong R \times _k S$.
\end{enumerate}
\end{thm}

\begin{proof}
(a) Assume that $\pi\colon A\to R$ is the retraction of the inclusion $R \subseteq A$. Setting $S = k \oplus \m_S$, where $\m_S = \Ker (\pi)$, and noting that the $R$-action $\ast$ on $\fm_S$ is given by the multiplication of the elements of $R$ and $\fm_S$ in $A$, it is straightforward to see that  $A \cong R \ltimes _k S$. The converse of part (a) is obvious.

(b) Let $\pi'\colon A\to S$ be a retraction of the inclusion $S\subseteq A$. By part (a), we can identify $A$ by $R\ltimes_k S$. Let $f\colon R \to S$ be the composition of the $k$-algebra homomorphism $\pi'$ and the inclusion map $R \subseteq A$ and note that the $R$-module structure on $\m_S$ is given via $f$. Therefore, it follows from Proposition~\ref{induced by f} that $R \ltimes _k S\cong R \times _k S$, as desired.
\end{proof}

The proof of Main Theorem will be given after further preparation that is included in the following two lemmas.

\begin{lem}\label{Lemma3}
Let $\varphi\colon (T, \m_T)\to (A, \m_A)$  be a $k$-algebra homomorphism. Assume that $T$ is complete and $A$ is separated, that is, $\bigcap _{n=1}^\infty \m_A^n =(0)$. If the equality $\m_TA=\m_A$ holds, then $\varphi$ is surjective. 
\end{lem}

\begin{proof}
Since  $A=k\oplus \m_A$, it follows from our assumption that
$$
\m_A = \m_T A= \varphi(\m_T) + \varphi(\m_T)\m_A.
$$
Therefore, for an element $x\in \fm_A$, there are a positive integer $r_1$ and elements $a_1, b_{1i}  \in \m_T$ and $x_{1i} \in \m_A$ for $1\leq i\leq r_1$ such that
$$
x = \varphi (a_1) +\sum _{i=1}^{r_1} \varphi(b_{1i})x_{1i}.
$$
By induction on $n$, we obtain an equality 
$$
x= \varphi (a_1+a_2+ \cdots +a_n) + \sum _{i=1} ^{r_n} \varphi (b_{ni}) x_{ni}
$$
in which $r_n$ is a positive integer, $a_j \in \m_T^j$ for all $1\leq j \leq n$ , $b_{ni} \in \m_T^n$, and  $x_{ni} \in \m_A$ for all $1 \leq i \leq r_n$. 
Since $T$ is complete, $\alpha = \sum _{i=1} ^{\infty} a_i$ is an element in $T$. 
Also, it follows from the equality
$$
x - \varphi (\alpha) = \varphi \left(\sum _{i=n+1} ^{\infty} a_i\right) + \sum _{i=1} ^{r_n} \varphi (b_{ni}) x_{ni}
$$
that $x -\varphi(\alpha) \in \m_A^n$ for all integers $n \geq 1$. 
Hence, by the fact that $A$ is separated we obtain the equality $x = \varphi (\alpha)$ in $A$, which means that $\varphi$ is surjective. 
\end{proof}


\begin{lem}\label{Lemma 4.}
Assume that $(R, \m_R)$ is a noetherian local $k$-algebra and $(T, \m_T)$ is a $k$-subalgebra of $(R, \m_R)$ such that $R$ is $T$-flat, and let $F$ be a complex of finitely generated free $R$-modules. If $\HH_n(F \otimes _T k)=0$ for an integer $n$, then $\HH_n(F)=0$.
\end{lem}

\begin{proof}
The proof is given in the following steps.\vspace{2mm}

\noindent \textbf{Step 1}: For all integers $i\geq 1$ we have $\HH_n\left(F \otimes _T \left(T/\m_T ^{i}\right)\right)= 0$. This statement is proved by induction on $i$. In fact, for an integer $i >1$, since $F$ is a $T$-flat complex, the short exact sequence 
$$
0 \to \m_T^{i -1} / \m_T^{i} \to T/\m_T^{i} \to T/\m_T^{i-1} \to 0
$$ 
of $T$-modules induces a short exact sequence 
$$
0 \to F \otimes _T  \left(\m_T^{i -1} / \m_T^{i}\right) \to F \otimes _T\left(T/\m_T^{i}\right) \to F \otimes _T \left(T/\m_T^{i-1}\right) \to 0
$$ 
of $T$-complexes. Note that the complex $F \otimes _T  \left(\m_T^{i -1} / \m_T^{i}\right)$ is a direct sum of $F \otimes _T k$ and therefore, $\HH_n\left(F \otimes _T  \left(\m_T^{i -1} / \m_T^{i}\right)\right)=0$. 
On the other hand, by inductive step we have $\HH_n\left(F \otimes _T\left(T/\m_T^{i-1}\right)\right) =0$. Hence, $\HH_n\left(F \otimes _T\left(T/\m_T^{i}\right)\right)=0$.\vspace{2mm}

\noindent \textbf{Step 2}: Denote by $B_n(F)$ and $Z_n(F)$ the $n$-th boundary set and the $n$-th cycle set of $F$, respectively. Then, for all integers $i\geq 1$ we have the containment $Z_n (F) \subseteq B_n (F) + \m_T^{i} F_n$. To prove this, note that for all integers $i\geq 1$ we have 
\begin{equation}\label{eq20251102a}
Z_n \left(F \otimes _T\left(T/\m_T^{i}\right)\right)=B_n\left(F \otimes _T\left(T/\m_T^{i}\right)\right) = \left( B_n(F) + \m_T^{i}F_n \right)  /\m_T^{i}F_n
\end{equation}
in which the first equality follows from Step 1. Now, let $z \in Z_n(F)$. Then, we have $z \otimes _T \left(T/\m_T^{i}\right)\in Z_n(F)\otimes_T\left(T/\m_T^{i}\right)\subseteq Z_n \left(F \otimes _T \left(T/\m_T^{i}\right)\right)$. Therefore, by~\eqref{eq20251102a} we conclude that $z \in  B_n(F) + \m_T^{i}F_n$.\vspace{2mm}

\noindent \textbf{Step 3}: We now prove that $\HH_n(F)=0$. For this, note that by Step 2 we have
$$
Z_n (F) \subseteq \bigcap _{i \geq 1} \left( B_n (F) + \m_T^{i}F_n \right) 
\subseteq \bigcap _{i \geq 1} \left( B_n (F) + \m_R^{i}F_n \right) 
= B_n(F)  
$$
in which the last equality holds since $B_n(F) \subseteq F_n$, where $F_n$ is a free module over the noetherian local ring $(R, \m_R)$. 
Therefore, we have  $Z_n(F)=B_n(F)$ and hence, the equality $\HH_n(F)=0$ holds. 
\end{proof}

\noindent \emph{Proof of Main Theorem.} The equivalence (ii) $\Longleftrightarrow$ (iii) has been proved in part (a) of Theorem~\ref{retraction theorem}.

(i)$\implies$(ii): Assuming that the $\overline{R}$-module $k$ is liftable to $R$, there exists an $R$-free complex $L$ such that $L \otimes _R \overline{R}$ is an $\overline{R}$-free resolution of $k$. Thus, it follows from the isomorphism $L \otimes _T k\cong L \otimes _R \overline{R}$ that $\HH_n(L \otimes _T k)=0$ for all integers $n\geq 1$. Therefore, by Lemma~\ref{Lemma 4.} we have $\HH_n(L)=0$ for all integers $n\geq 1$. Setting $\HH_0(L) = R/J$ for an ideal $J$ of $R$, we have that $L$ is an $R$-free (hence, $T$-flat) resolution of $R/J$. Since $R/J$ is a lifting module of $k$, we have $R/J\otimes_R\overline{R}\cong k$ and hence, $\m_{R/J}$ can be identified by $\m_T (R/J)$. By Lemma~\ref{Lemma3}, the composition $T \xra{\varphi} R \twoheadrightarrow R/J$ is surjective. On the other hand, since $L \otimes _T k$ is acyclic, we have $\Tor _i^T(R/J, k) =0$ for all integers $i \geq 1$. Therefore,  $R/J$ is $T$-flat and we conclude that $T \xra{\varphi} R \twoheadrightarrow R/J$ is an isomorphism, meaning that $\varphi$ has a retraction. 

(ii)$\implies$(i): Let $\pi\colon R\to T$ be a retraction of $\varphi$. Then, $T$ is an $R$-module via $\pi$ and we can consider an $R$-free resolution
$F=\cdots \to F_1 \to F_0 \to 0$ of $T$. Since each $F_i$ is $T$-flat, $F \otimes _R \overline{R} \cong F \otimes _T k$ is an $\overline{R}$-free resolution of $T\otimes_R \overline{R}\cong k$. Thus, $F$ is a lifting of the $\overline{R}$-module $k$ to $R$, as desired.
\qed\vspace{2mm}

The following result is an immediate consequence of Main Theorem.

\begin{cor}\label{cor20251018a}
Assume that $(R, \m_R)$ is a noetherian complete local $k$-algebra, and let $\overline{R} = R/\underline{x}R$, where $\underline{x} \in \m_R$ is a regular sequence in $R$. Then, the following statements are equivalent: 
\begin{enumerate}[\rm(i)]
\item
The $\overline{R}$-module $k$ is liftable to $R$; 
\item
The inclusion $k[\![\underline{x}]\!] \subseteq R$ has a retraction;
\item
There is a $k$-subalgebra $(S, \m_S)$ of $(R, \m_R)$ such that $R$ admits a semi-fiber product decomposition $R \cong k[\![\underline{x}]\!] \ltimes _k S$.
\end{enumerate}
More precisely, under these equivalent conditions, $S=k\oplus \fm_S$, where $\fm_S$ is the kernel of the retraction of the inclusion $k[\![\underline{x}]\!] \subseteq R$.
\end{cor}

\begin{proof}
Since $\underline{x}$ is a regular sequence in $R$, the inclusion map $k[\![\underline{x}]\!]\subseteq R$ is flat; see for instance~\cite[Exercise 18.18]{Eisenbud}. Therefore, the assertion follows from the proof of Theorem~\ref{retraction theorem} (a) and Main Theorem.
\end{proof}

\begin{para}
Consider the setting of Corollary~\ref{cor20251018a}. According to our discussion in~\ref{para20251104a}, if $\Ext^2_{\overline{R}}(k,k)=0$, i.e., $\overline{R}$ is a regular ring of Krull dimension at most $1$, then the equivalent conditions of Corollary~\ref{cor20251018a} hold.
\end{para}

We conclude this paper with the following result which is a variation of Corollary~\ref{cor20251018a}. Note that in this result, unlike Corollary~\ref{cor20251018a}, we are working with a non-regular element.

\begin{cor}\label{cor20250928a}
Assume that $(R, \m_R)$ is a noetherian local $k$-algebra (not necessarily complete). Let $n$ be a positive integer, $x$ be an element in $\fm_R$ such that $(0:x) = x^nR\neq (0)$, and $\overline{R}= R/xR$. Then, the following are equivalent: 
\begin{enumerate}[\rm(i)]
\item
The $\overline{R}$-module $k$ is liftable to $R$; 
\item
Considering the $k$-subalgebra $(k[x], (x))$ of $(R, \m_R)$, the inclusion $k[x] \subseteq R$ has a retraction;
\item
There is a $k$-subalgebra $(S, \m_S)$ of $(R, \m_R)$ such that $R$ admits a semi-fiber product decomposition $R \cong k[x] \ltimes _k S$.
\end{enumerate}
More precisely, under these equivalent conditions, $S=k\oplus \fm_S$, where $\fm_S$ is the kernel of the retraction of the inclusion $k[x] \subseteq R$.
\end{cor}

\begin{proof}
Let $T=k[X]/(X^{n+1})$ and consider the $k$-algebra homomorphism $T\to R$ obtained by composing the maps $T\xra{\cong} k[x]$ and $k[x]\hookrightarrow R$. Note that the sequence
\begin{equation}\label{eq20251107a}
\cdots\xra{X} T\xra{X^n} T\xra{X} T\to k\to 0
\end{equation}
is exact. Thus, the assumption $(0:x) = x^nR$ implies that the sequence $$\cdots\xra{x} R\xra{x^n} R\xra{x} R\to \overline{R}\to 0$$ obtained by applying the functor $R\otimes_T-$ on the exact sequence~\eqref{eq20251107a} is also exact. Hence, $\Tor^T_i(R,k)=0$ for all integers $i\geq 1$, that is, $R$ is $T$-flat. Now, the assertion follows from the proof of Theorem~\ref{retraction theorem} (a) and Main Theorem.
\end{proof}





\providecommand{\bysame}{\leavevmode\hbox to3em{\hrulefill}\thinspace}
\providecommand{\MR}{\relax\ifhmode\unskip\space\fi MR }
\providecommand{\MRhref}[2]{%
  \href{http://www.ams.org/mathscinet-getitem?mr=#1}{#2}
}
\providecommand{\href}[2]{#2}

\end{document}